%% file: harmonic.tex
\documentclass [a4paper,10pt,article,oneside]{memoir}
\usepackage[utf8]{inputenc}
\usepackage{amsmath}
\usepackage{amsfonts,mathrsfs,paralist,amssymb,bm,amsthm}
\usepackage{graphicx,color,paralist}
\usepackage{cite}
\usepackage{url}
\usepackage{mathrsfs}
\usepackage[all]{xy}
\xyoption{poly}
\setsecnumformat{\csname the#1\endcsname---}

\setsecheadstyle{\centering\Large\bfseries\sffamily}
\setsubparaheadstyle{\itshape}
\newtheoremstyle{thm}
     {.8ex plus .3ex minus .1ex}
     {.8ex plus .3ex minus .1ex}
     {\itshape}
     {}
     {\bfseries}
     {---}
     {0em}
     {$\bullet$\hbox{\ }#1\hbox{\ }#2}

\theoremstyle{thm}

\newtheorem{definition}{Definition}[section]
\newtheorem{theorem}[definition]{Theorem}
\newtheorem{lemma}[definition]{Lemma}
\newtheorem{proposition}[definition]{Proposition}
\newtheorem{corollary}[definition]{Corollary}

\newtheoremstyle{note}
     {1ex plus .3ex minus .1ex}
     {1ex plus .3ex minus .1ex}
     {}
     {}
     {\itshape}
     {.}
     {1em}
     {}
\theoremstyle{note}
\newtheorem*{remark}{Remark}

\newcommand{\tq}{\;:\;}

\newcommand{\indep}{\parallel}
\newcommand{\up}[1]{\,\uparrow #1}

\newcommand{\ZZ}{\mathbb{Z}}
\newcommand{\pr}{\mathbb{P}}
\newcommand{\esp}{\mathbb{E}}

\newcommand{\B}{\mathcal{B}}
\newcommand{\C}{\mathscr{C}}
\newcommand{\Cstar}{\mathfrak{C}}

\newcommand{\RR}{\mathbb{R}}

\newcommand{\FFF}{\mathfrak{F}}

\newcommand{\slgb}{\mbox{$\sigma$-al}\-ge\-bra}

\newcommand{\pas}{\text{$\pr$-a.s.}}

\newcommand{\M}{\mathcal{M}}

\newcommand{\Mbar}{\overline{\M}}

\renewcommand{\hbar}{\widetilde{h}}

\newcommand{\R}{\mathscr{R}}

\renewcommand{\H}{\mathcal{H}}

\newcommand{\height}{\tau}

\newcommand{\MH}{\mathsf{MH}}
\renewcommand{\B}{\partial}

\newcommand{\un}[1]{\mathbf{1}_{\{#1\}}}
\newsavebox{\mycaptionbox}
\savebox{\mycaptionbox}{\parbox[t]{.8\textwidth}{\raggedright\slshape A trace $u$ with Cartier-Foata normal form
    $c_1\to\ldots\to c_n$ and a generic element $x\in\M(u)$ with
    Cartier-Foata normal form $(c_1\cdot\gamma_1)\to\ldots\to(c_n\cdot\gamma_n)$\,.}}

\setsecnumdepth{subsection}

\newlength{\firstl}
\newlength{\secondl}
\newlength{\finall}

\newcommand{\firstc}{A Graded M\"obius Transform}
\newcommand{\secondc}{and its Harmonic Interpretation}
\newcommand{\format}[1]{\makebox[\finall][s]{#1}} \newsavebox{\firstb}
\newsavebox{\secondb}

\newcommand{\lfaast}{%
  \savebox{\firstb}{\firstc}%
  \savebox{\secondb}{\secondc}%
  \setlength{\firstl}{\wd\firstb}%
  \setlength{\secondl}{\wd\secondb}%
  \ifnum\firstl>\secondl\setlength{\finall}{\firstl}\else
  \setlength{\finall}{\secondl}\fi\par
  \format{\firstc}\\
\format{\secondc}%
  }

\everymath{\normalfont}

\begin{document}
\mainmatter
\strut\vspace{-4em}
\begin{center}
\huge\bfseries\sffamily
\lfaast\par
\normalfont
\end{center}

\medskip

\begin{center}
\begin{tabular}{c}
  \Large\sffamily Samy Abbes\\
  \small University Paris Diderot -- Paris~7\\
  \small CNRS Laboratory PPS (UMR 7126)\\
  \small Paris, France\\
  \small
  \ttfamily\footnotesize  samy.abbes@univ-paris-diderot.fr
\end{tabular}
\end{center}


\input{intro-groups_harmonic.tex}

\input{main_harmonic.tex}

%% file: intro-groups_harmonic.tex
\begin{abstract}
  We give a graded version of the M\"obius inversion formula in the
  framework of trace monoids. The formula is based on a graded
  version of the M\"obius transform, related to the notion of height
  deriving from the Cartier-Foata normal form of the elements of a
  trace monoid.

  Using the notion of Bernoulli measures on the boundary of a trace
  monoid developed recently, we study a probabilistic interpretation
  of the graded inversion formula.  We introduce M\"obius harmonic
  functions for trace monoids and obtain an integral representation
  formula for them, analogous to the Poisson formula for harmonic
  functions associated to random walks on trees.
\end{abstract}


{\raggedright \textbf{AMS Subject Classification:} 20E05 60J45 60B15
  06F05 

}

\section*{Introduction}

A random walk on a regular oriented tree of finite degree, which we
identify with a free monoid~$\Sigma^*$, is specified by a probability
distribution $f$ on the finite set $\Sigma$ of generators. In turn,
the law of the trajectories of the random walk is a Bernoulli measure
on the space of $\Sigma$-valued infinite sequences, which identifies
with the boundary at infinity $\B\Sigma^*$ of the tree. Recall that,
if we denote by~$\up x$\,, for $x$ ranging over~$\Sigma^*$, the set of
infinite sequences of which $x$ is a prefix, then Bernoulli measures
on $\B\Sigma^*$ are characterized by the multiplicative property:
$\pr\bigl(\up(xy)\bigr)=\pr(\up x)\pr(\up y)$\,, valid for all
$x,y\in\Sigma^*$\,.

Let $P$ denote the Markov operator acting on real valued functions
defined on~$\Sigma^*$, and such that
$P\lambda(x)=\sum_{a\in\Sigma}f(a)\lambda(xa)$ for all $x\in\Sigma^*$
and for all functions $\lambda:\Sigma^*\to\RR$\,. Harmonic functions
relative to the pair $(\Sigma^*,\pr)$ are those functions
$\lambda:\Sigma^*\to\RR$ such that $P\lambda=\lambda$, hence in the
kernel of the discrete Laplace operator $\Delta=I-P$. It is well known
that bounded harmonic functions are in a linear and isometric
one-to-one correspondence with measurable and essentially bounded
real valued functions defined on~$\B\Sigma^*$\,, through the Poisson
formula:
\begin{equation}
  \label{eq:1}
  \forall x\in\Sigma^*\quad\lambda(x)=\frac1{\pr(\up x)}\int_{\up
    x}\varphi(\xi)\,d\pr(\xi)\,,\quad \varphi\in L^\infty(\B\Sigma^*)\,.
\end{equation}

If $\lambda$ is bounded harmonic, then $\varphi$ is obtained as the
\pas\ limit of the bounded martingale
$(\lambda(X_n),\FFF_n)_{n\geq1}$\,, where $(X_n)_{n\geq1}$ is the
random walk on $\Sigma^*$ and $\FFF_n=\sigma\langle
X_1,\ldots,X_n\rangle$\,.

In this paper, we study the notion of harmonicity in a slightly
different framework. First, instead of considering a free
monoid~$\Sigma^*$\,, we allow some generators to commute with each
other, and consider thus free partially commutative monoids, usually
called trace monoids~\cite{diekert90}, and referred to in the
literature---also in the case of groups---as to heap
monoids~\cite{viennot86}, locally free monoids with respect to a
graph~\cite{VNBi,malyutin05}, partially commutative
monoids~\cite{cartier69}, graph monoids~\cite{fisher89}. Hence a trace
monoid $\M$ is a finitely presented monoid of the form
$\M=\Sigma^*/\R$\,, where $\R$ is the congruence relation generated by
pairs of the form $(ab,ba)$, for $(a,b)$ ranging over a given
symmetric and irreflexive relation $I$ on~$\Sigma$, called an
independence relation. The elements of a trace monoids are called
traces. Much of the above framework for free monoids can be transposed
to trace monoids. In particular, there is a natural notion of infinite
trace---which corresponds to a simplified version of the stable normal
form for infinite words in the sense of~\cite{vershik00}. The boundary
at infinity $\B\M$ of the monoid~$\M$ is defined as the set of
infinite traces. An elementary cylinder~$\up u$, for $u\in\M$, is
defined as the subset of those infinite traces of which $u$ is a
prefix: $\up u=\{\xi\in\B\M\tq u\leq \xi\}$\,, and
$\FFF=\sigma\langle\up u\tq u\in\M\rangle$ is the \slgb\ that
equips~$\B\M$\,.

Second, instead of considering random walks on a trace monoid, we
directly consider Bernoulli measures on the boundary of the monoid. A
Bernoulli measure on a trace monoid is defined as a probability
measure $\pr$ on $(\B\M,\FFF)$, such that the following multiplicative
property holds:
\begin{equation}
  \label{eq:5}
  \forall u,v\in\M\quad \pr\bigl(\up(u\cdot v)\bigr)=\pr(\up u)\pr(\up v)\,.
\end{equation}

In a recent work with J.~Mairesse~\cite{abbesmair14}, we have
conducted a thorough study of Bernoulli measures for trace monoids, by
showing how to characterize them through probabilistic parameters and
by giving an explicit construction of them by means of the
combinatorial structure of the monoid. Connecting them with more
familiar objects usually found in this journal, Bernoulli measures can
be seen as weighted Patterson-Sullivan measures. However, they are
\emph{not} given as the law of entrance of a random walk into the
boundary at infinity of the monoid, except if the trace monoid reduces
to a free monoid, which corresponds to the empty independence relation
$I=\emptyset$\,. Hence, the framework found in
\cite{VNBi,vershik00,malyutin05} for instance does not apply for
Bernoulli measures.

The main topic of this paper is the notion of harmonicity in the
framework of trace monoids equipped with Bernoulli measures. In order
to obtain a dual representation between bounded functions on the
boundary~$\B\M$ on the one hand, and functions defined on $\M$
invariant with respect to a certain linear operator, consider a pair
$(\M,\pr)$ where $\M$ is a trace monoid and $\pr$ is a Bernoulli
measure on~$\B\M$, and let $\varphi\in L^\infty(\B\M)$. We define an
associated function \mbox{$\lambda:\M\to\RR$} by:
\begin{equation}
  \label{eq:6}
  \forall u\in\M\quad \lambda(u)=\frac1{\pr(\up u)}\int_{\up u}\varphi(\xi)\,d\pr(\xi)\,.
\end{equation}

Then we observe that the function $\lambda$ satisfies the following relation:
\begin{equation}
  \label{eq:18}
  \forall u\in\M\quad\sum_{c\in\C}(-1)^{|c|}f(c)\lambda(u\cdot c)=0\,,
\end{equation}
where $\C$ denotes the set of cliques of the finite graph
$(\Sigma,I)$\,, and $f:\M\to\RR$ is the multiplicative function on
$\M$ defined by $f(u)=\pr(\up u)$\,. Because of the deep relationship
between the expression~(\ref{eq:18}) and the M\"obius polynomial of
the pair~$(\Sigma,I)$, defined by
$\mu_\M(X)=\sum_{c\in\C}(-1)^{|c|}X^{|c|}$\,, we call the operator
$\Delta$ acting on functions $\lambda:\M\to\RR$ by:
\begin{equation}
  \label{eq:45}
  \forall
  u\in\M\quad\Delta\lambda(u)=\sum_{c\in\C}(-1)^{|c|}\lambda(u\cdot c)\,,
\end{equation}
the M\"obius-Laplace operator on~$\M$\,; functions in the kernel
of~$\Delta$, we call M\"obius harmonic. 

Decomposing cliques according to their size, $\Delta$~writes as:
\begin{equation}
  \label{eq:46}
  \Delta=I-P\,,\quad P\lambda(u)=\sum_{a\in\Sigma}f(a)\lambda(u\cdot
  a)-\sum_{c\in\C\tq|c|\geq2}(-1)^{|c|}f(c)\lambda(u\cdot c)\,.
\end{equation}

Hence, as it turns out, M\"obius harmonic functions do not have an
obvious interpretation as invariant functions with respect to a Markov
operator; for $P=I-\Delta$ is not a positive operator, unless the
trace monoid reduces to the free monoid~$\Sigma^*$\,. This
contrasts with the case of random walk on
trees~\cite{cartier72,mouton00}, but also on more general hyperbolic
structures~\cite{kaimanovich96:,kaimanovich00}.

Nevertheless, the correspondence through the Poisson formula still
holds: the main result of this paper is the existence, for every
bounded M\"obius harmonic function $\lambda:\M\to\RR$, of a unique
essentially bounded function $\varphi\in L^\infty(\B\M)$ on the
boundary, such that formula~(\ref{eq:6}) holds for the
pair~$(\lambda,\varphi)$.

Our technique of proof resembles to some extent the usual technique,
for trees for instance. However, starting from a bounded M\"obius
harmonic function $\lambda:\M\to\RR$, obtaining the martingale
$(Y_n,\FFF_n)_{n\geq1}$ which converges \pas\ towards the adequate
function $\varphi\in L^\infty(\B\M)$ is more involved than usual. In
order to put in motion the martingale machinery, we rely on a
generalization of the M\"obius transform, as popularized by
G.-C.{}~Rota~\cite{rota64}. The original M\"obius inversion formula,
first formulated for integers, was shown by Rota to be a particular
case of a formula best formulated in the incidence algebra associated
to a general class of partial orders. For a trace monoid, the M\"obius
inversion formula writes as follows, for any function $F:\C\to\RR$
defined on the set of cliques of the graph~$(\Sigma,I)$:
\begin{equation*}
  \forall c\in\C\qquad
\begin{aligned}[t]
 F(c)&=\sum_{c\in\C\tq c'\geq c}H(c')\,,&
H(c)&=\sum_{c'\in\C\tq c'\geq c}
(-1)^{|c'|-|c|}F(c')\,. 
\end{aligned}
\end{equation*}

The extended M\"obius inversion formula that we prove in this paper
holds for functions $F:\M\to\RR$ defined on $\M$ rather than on~$\C$
only. We show that it is an adequate tool to obtain the integral
representation formula for bounded M\"obius harmonic functions.

The extended M\"obius transform on which the new inversion formula is
based, makes use of the natural graded structure attached to elements
of a trace monoid. Indeed, traces can be put in a normal form---the
Cartier-Foata normal form. The graded structure of the trace monoid
$\M$ is the partition of $\M$ into traces with a fixed number of
elements in their Cartier-Foata normal form. This number is called the
height of a trace. Observe that the height does not correspond to the
geodesic distance between an element of the monoid and the identity of
the monoid.

Without any doubt, the graded M\"obius inversion formula that we state
should be valid for more general ``graded partial orders'', of which
braid monoids should typically be an instance. We felt however that
the trace monoid case was already non trivial, yet it allows for a
thorough presentation of the main ideas.

\paragraph*{Organization of the paper.}
\label{sec:organization-paper}

In Section~\ref{sec:comb-trace-mono}, we recall elements on the
Combinatorics of trace monoids, following~\cite{cartier69}, and we
provide essential information on the boundary of trace monoids and on
Bernoulli measures, following~\cite{abbesmair14}. The contributions of
the paper appear in
Sections~\ref{sec:prel-comb-results}--\ref{sec:non-negative-mobius}.
In Section~\ref{sec:prel-comb-results}, we introduce the graded
M\"obius transform and we prove the associated inversion
formula. Bounded M\"obius harmonic functions are the topic of
Section~\ref{sec:mobi-harm-funct}.
Section~\ref{sec:non-negative-mobius} deals with examples of
non-negative and unbounded M\"obius harmonic functions and introduces
the analogous of the Green and the Martin kernel. Finally,
Section~\ref{sec:conclusion} concludes the paper.


%% file: main_harmonic.tex
\section{Trace Monoids and Bernoulli Measures on their Boundary}
\label{sec:comb-trace-mono}

This section collects material on trace monoids, introduces the
boundary of a trace monoid and associated Bernoulli measures.

\paragraph*{Trace monoids.}
\label{sec:trace-monoids}

Let $\Sigma$ be a finite set, referred to as to the
\emph{alphabet}. Elements of $\Sigma$ are called \emph{letters}. By
convention, we only consider throughout the paper alphabets of
cardinality~$>1$.  An \emph{independence relation} $I$ on $\Sigma$ is
a binary relation on~$\Sigma$, symmetric and irreflexive. Let
$\Sigma^*$ be the free monoid on~$\Sigma$, and let $\R$ be the
congruence relation on $\Sigma^*$ generated by $\{(ab,ba)\tq (a,b)\in
I\}$. The quotient monoid $\M(\Sigma,I)=\Sigma^*/\R$ is called a
\emph{trace monoid}, and its elements are called
\emph{traces}. Classical references on trace monoids
are~\cite{cartier69,viennot86,diekert90}.

The \emph{dependence relation} associated to an independence relation
$I$ on $\Sigma$ is defined by $D=(\Sigma\times\Sigma)\setminus I$.
The trace monoid $\M(\Sigma,I)$ is said to be \emph{irreducible}
whenever the dependence relation $D$ makes the graph $(\Sigma,D)$
connected.

Put $\M=\M(\Sigma,I)$. For any trace $u\in\M$, any two representative
words of $u$ have the same length, which defines the \emph{length}
$|u|$ of~$u$. Let $0$ denote the empty trace, image in $\M$ of the
empty word, and let $\cdot$ denote the concatenation of traces. The
\emph{prefix relation} on~$\M$, denoted~$\leq$, is defined by:
\begin{equation*}
  \forall u,v\in\M\quad u\leq v\iff \exists w\in\M\quad v=u\cdot w.
\end{equation*}
This is a partial order relation on~$\M$. 

Trace monoids are known to be right and left cancellative, meaning:
\begin{equation}
  \label{eq:36}
  \forall x,x'\in\M\quad
\forall y,z\in\M\quad
y\cdot x\cdot z=y\cdot x'\cdot z\implies x=x'.
\end{equation}

\paragraph{Boundary and elementary cylinders. Bernoulli measures.}
\label{sec:boundary}

Let $\H$ denote the set of non-decreasing sequences $(x_n)_{n\geq0}$
in~$\M$, those sequences such that $x_n\leq x_{n+1}$ for all integers
$n\geq0$. We identify any two sequences $x=(x_n)_{n\geq0}$ and
$y=(y_n)_{n\geq0}$ in $\H$ such that $x\preccurlyeq y$ and
$y\preccurlyeq x$, where we have defined relation $\preccurlyeq$ as
follows:
\begin{equation*}
  \forall x,y\in\H\quad
x\preccurlyeq y\iff\forall n\geq0\quad\exists m\geq0\quad x_n\leq y_m\,.
\end{equation*}

Let $\Mbar$ denote the quotient set $\H/\equiv$, with $x\equiv y\iff
x\preccurlyeq y\wedge y\preccurlyeq x$\,. Then $\Mbar$ is just the
collapse of the pre-ordered set $(\H,\preccurlyeq)$, and as such, is
is equipped with a partial ordering relation~$\leq$. There is a
canonical injection $\M\to\Mbar$ which respects the ordering, and
which maps each trace $u\in\M$ to the image in $\Mbar$ of the constant
sequence $(u,u,\cdots)$. This justifies that elements of $\Mbar$ are
called \emph{generalized traces}.  Elements of
$\B\M=\Mbar\setminus\M$\,, identifying $\M$ with its image
in~$\Mbar$\,, are called \emph{infinite traces}. The set $\B\M$ is
called the \emph{boundary} of~$\M$~\cite{abbes08,abbesmair14}. By
construction, any sequence $(x_n)_{n\geq1}$ of elements of $\Mbar$
which is non decreasing has a least upper bound in~$\Mbar$\,,
denoted $\bigvee_{n\geq1}x_n$\,.

For each trace $u\in \M$, the \emph{elementary cylinder of base $u$}
is the following subset of~$\B\M$:
\begin{equation}
  \label{eq:7}
  \up u=\{\xi\in\B\M\tq u\leq \xi\}\,.
\end{equation}

We denote by $\FFF$ the \slgb\ on $\B\M$ generated by the countable
collection of elementary cylinders. We say that a probability measure
$\pr$ on $(\B\M,\FFF)$ is a \emph{Bernoulli
  measure}~\cite{abbesmair14} if $\pr(\up u)>0$ for all $u\in\M$, and
if:
\begin{equation}
  \label{eq:8}
  \forall u,v\in\M\quad \pr\bigl(\up(u\cdot v)\bigr)=\pr(\up u)\pr(\up v)\,.
\end{equation}

\paragraph{Cliques and Cartier-Foata decomposition.}
\label{sec:cliq-cart-foata}

An \emph{independence clique}, or a \emph{clique} for short, of a pair
$(\Sigma,I)$, is defined as a subset $c\subseteq\Sigma$ of the
alphabet, such that any two distinct letters $a,b\in c$ satisfy
$(a,b)\in I$. If $c=\{a_1,\ldots,a_n\}$ is a clique, then the product
$a_1\cdot\ldots\cdot a_n\in\M$ is independent of the chosen
enumeration of~$c$. Therefore cliques identify with their images
in~$\M$. The order of cliques in $\M$ corresponds to the inclusion
ordering on subsets of~$\Sigma$. We denote by $\C$ the set of cliques
associated to~$\M$, and by $\Cstar=\C\setminus\{0\}$ the set of non
empty cliques.

Let $c,c'$ be two cliques. We say that $(c,c')$ is
\emph{Cartier-Foata} admissible, denoted by $c\to c'$\,, if for every
letter $b\in c'$, there exists a letter $a\in c$ such that
$(a,b)\notin I$. In the heap of pieces interpretation of
Viennot~\cite{viennot86}, this corresponds to the letter $b$ being
blocked from below by the letter~$a$. It is known that for every non
empty trace $u\in\M$, there exists a unique integer $n\geq1$ and a
unique sequence of non empty cliques $(c_i)_{1\leq i\leq n}$ such
that:
\begin{align*}
  \forall i\in\{1,\ldots,n-1\}\quad c_i&\to c_{i+1}\,,&
u&=c_1\cdot\ldots\cdot c_n\,.
\end{align*}

This sequence $(c_i)_{1\leq i\leq n}$ is called the
\emph{Cartier-Foata decomposition} of~$u$~\cite{cartier69}. The
integer $n$ is called the \emph{height} of~$u$, we denote it by
$n=\height(u)$.

The Cartier-Foata decomposition extends to infinite traces: for every
infinite trace $\xi\in\B\M$, there exists a unique infinite sequence
$(c_i)_{i\geq1}$ of non empty cliques~\cite[Lemma~8.4]{abbesmair14} such that:
\begin{align*}
  \forall i\geq1\quad c_i&\to c_{i+1}\,,&
\xi&=\bigvee_{n\geq1}(c_1\cdot\ldots\cdot c_n)\,.
\end{align*}
The infinite sequence $(c_i)_{i\geq1}$ is called the \emph{Cartier-Foata
decomposition} of~$\xi$.

\paragraph{M\"obius transform. M\"obius
  polynomial. Characterization of Bernoulli measures.}
\label{sec:mobi-transf-mobi}

The M\"obius polynomial of $\M$ is $\mu_\M(X)\in\ZZ[X]$ defined by:
\begin{equation}
  \label{eq:11}
  \mu_\M(X)=\sum_{c\in\C}(-1)^{|c|}X^{|c|}\,.
\end{equation}
It is also referred to in the literature as to the clique polynomial
of~$(\Sigma,I)$, and coincides up to a change of variable with the
independence polynomial~\cite{levit05} of the graph $(\Sigma,D')$,
where $D'=\bigl((\Sigma\times\Sigma)\setminus
I\bigr)\setminus\{(x,x)\tq x\in\Sigma\}$\,.

It is known that $\mu_\M(X)$ has a unique root $p_0$ of smallest
modulus, and that
$p_0\in(0,1)$~\cite{goldwurm00,krob03,csikvari13}. If $\M$ is
irreducible, then~\cite[Th.~5.1]{abbesmair14} there is a unique
Bernoulli measure $\pr$ on $(\B\M,\FFF)$ such that $\pr(\up
u)=p_0^{|u|}$ for all $u\in\M$\,.

More generally, let $f:\M\to\RR$ be a function.  We say that $f$ is a
\emph{valuation} if $f(u\cdot v)=f(u)f(v)$ holds for all $u,v\in\M$.
Let $f$ be a positive valuation on~$\M$, and assume that $\M$ is
irreducible.  Let $h:\C\to\RR$ be the \emph{M\"obius transform} of the
restriction~$f|_\C$\,, defined by:
\begin{equation}
  \label{eq:10}
  \forall c\in\C\quad h(c)=\sum_{c'\in\C\tq c'\geq c}(-1)^{|c'|-|c|}f(c')\,.
\end{equation}

Then $f(u)=\pr(\up u)$ for some Bernoulli measure $\pr$ if and only if
the following two conditions \cite[Th.~3.3]{abbesmair14} are satisfied:
\begin{align}
  \label{eq:12}
  h(0)&=0\,,&
\forall c\in\Cstar\quad h(c)&>0\,.
\end{align}
The conditions in~(\ref{eq:12}) consist in a polynomial equality, and
several polynomial inequalities, involving only a finite number of
parameters, namely the numbers $f(a)$ for $a$ ranging over~$\Sigma$,
and that characterize the valuation~$f$.

Note that the M\"obius transform $h:\C\to\RR$ defined in~(\ref{eq:10})
makes sense for any function $f:\C\to\M$, and not only for the
restriction to $\C$ of a valuation defined on~$\M$.

Let $\pr$ be a Bernoulli measure on~$\B\M$, and let $f(u)=\pr(\up u)$
be the associated valuation. The sequence of non empty cliques
$\bigl(C_i(\xi)\bigr)_{i\geq1}$ which appear in the Cartier-Foata
decomposition of an infinite trace $\xi\in\B\M$, forms a sequence of
random variables. We know~\cite[Th.~4.1]{abbesmair14} that, under the
measure~$\pr$, the sequence $(C_i)_{i\geq1}$ is a time-homogeneous
Markov chain, which satisfies the following property, for every finite
sequence of non empty cliques $c_1\to\ldots\to c_n$:
\begin{equation}
  \label{eq:13}
  \pr(C_1=c_1,\ldots,C_n=c_n)=f(c_1)\cdots f(c_{n-1})h(c_n)\,.
\end{equation}

The law of~$C_1$, which is the initial distribution of the chain,
coincides with the restriction~$h|_{\Cstar}$\,. The transition matrix
$P=(P_{c,c'})_{(c,c')\in\Cstar\times\Cstar}$ is given by:
\begin{align}
\label{eq:26}
  P_{c,c'}&=
  \begin{cases}
    0,&\text{if $\neg(c\to c')$}\,,\\
h(c')/g(c),&\text{if $c\to c'$},
  \end{cases}
& g(c)=\sum_{c'\in\Cstar\tq c\to c'}h(c')\,.
\end{align}

Furthermore, as a consequence of the assumption $h(0)=0$ stated
in~(\ref{eq:12}), one has~\cite[Prop.~10.3]{abbesmair14}:
\begin{equation}
  \label{eq:30}
  \forall c\in\C\quad h(c)=f(c)g(c)\,.
\end{equation}

\paragraph{Ordering and Cartier-Foata decomposition.}
\label{sec:order-cart-foata}

We recall some results related to the Cartier-Foata decomposition of
traces and of infinite traces.

We still denote by $C_n(\xi)$ the $n^{\text{th}}$~clique in the
Cartier-Foata decomposition of an infinite trace $\xi\in\B\M$,
dropping the dependency with respect to $\xi$ when seeing $C_n$ as a
random variable defined on~$\B\M$. Let $u\in\M$ be a trace, of
Cartier-Foata decomposition $c_1\to\ldots\to c_n$\,. Then one has
\cite[Prop.~8.5]{abbesmair14} the
following equalities of subsets of~$\B\M$, putting
$v=c_1\cdot\ldots\cdot c_{n-1}$\,:
\begin{gather}
\label{eq:14}
\up u=\{\xi\in\B\M\tq C_1\cdot\ldots\cdot C_n\geq u\}\\
\label{eq:15}
\{\xi\in\B\M\tq C_1=c_1,\ldots,C_n=c_n\}=\up
u\setminus\Bigl(\bigcup_{c\in\C\tq c>c_n}\up (v\cdot c)\Bigr)
\end{gather}

Finally, define two cliques $c$ and $c'$ to be \emph{parallel}
whenever $c\times c'\subseteq I$, denoted by $c\indep c'$\,. If
$u,v\in\M$ are two traces, with $u=c_1\to\ldots\to c_n$ and
$v=d_1\to\ldots \to d_p$ their Cartier-Foata decompositions, then
\cite[Lemma~8.1]{abbesmair14} $u\leq v$ if and only if $n\leq p$, and
there exists cliques $\gamma_1,\ldots,\gamma_n$ such that:
\begin{align}
  \label{eq:20}
& \text{$d_i=c_i\cdot\gamma_i$ for $i\in\{1,\ldots,n\}$; and}\\
\label{eq:21}
&\text{$\gamma_i\indep c_j$ for all $i,j\in\{1,\ldots,n\}$ with $i\leq j$\,.}
\end{align}
The sequence of cliques $(\gamma_i)_{1\leq i\leq n}$ as above is
unique. An illustration is given in Figure~\ref{fig:sqpokqlmq} in next
section.

\section{The Graded M\"obius Transform}
\label{sec:prel-comb-results}

The M\"obius inversion formula, which holds for general classes of
partial orders~\cite{rota64}, takes the following form for trace
monoids: for any function $f:\C\to\RR$, with M\"obius transform
$h:\C\to\RR$ defined as in~(\ref{eq:10}), the function $f$ can be
retrieved from its transform through the formula
(see~\cite[Prop.~10.1]{abbesmair14} for a justification):
\begin{equation}
  \label{eq:17}
  \forall c\in\C\quad
f(c)=\sum_{c'\in\C\tq c'\geq c}h(c')\,.
\end{equation}

In this section, we give a generalization of~(\ref{eq:17}). For this,
we introduce the graded M\"obius transform of functions with
domain~$\M$, instead of $\C$ only. The graded M\"obius transform uses
the partition of $\M$ according to the height of traces.

The probabilistic interpretation of the corresponding inversion
formula will be the topic of next section.

\begin{definition}
  \label{def:1}
Let $\M$ be a trace monoid, and let $F:\M\to\RR$ be a function. The
\emph{graded M\"obius transform} of $F$ is the function
$H:\M\to\RR$ defined as follows. For $u\in\M$ a generic non empty
trace, denote by $c$ the last clique of the  Cartier-Foata decomposition
of~$u$. Let also $v$ be the unique trace such that $u=v\cdot c$. Then
define $H(u)$ by:
\begin{equation}
  \label{eq:9}
  H(u)=\sum_{c'\in\C\tq c'\geq c}(-1)^{|c'|-|c|}F(v\cdot c')\,.
\end{equation}
Define also $H(0)=\sum_{c\in\C}(-1)^{|c|}F(c)$\,.
\end{definition}

How to retrieve $F$ from its graded M\"obius transform $H$ is stated
in next result. Recall that the height $\height(u)$ of a trace
$u\in\M$ is the number of cliques in its Cartier-Foata decomposition,
with the convention $\tau(0)=0$. For each trace $u\in\M$, we put:
\begin{align*}
\text{for $u\neq0$\,:}\quad\M(u)&=\{x\in\M\tq \tau(x)=\tau(u)\wedge u\leq x\}\,,\\
\M(0)&=\C\,.
\end{align*}

See an illustration in Figure~\ref{fig:sqpokqlmq}.
\def\piece#1{\save"G";"G"+(24,0)**@{-};"G"+(24,12)**@{-};"G"+(0,12)**@{-};"G"**@{-},\restore}
\begin{figure}
\centerline{\xy
<.1em,0em>:
(0,0)="G",
"G"+(12,6)*{\gamma_1},
"G";"G"+(24,0)**@{--};"G"+(24,12)**@{--};"G"+(0,12)**@{--};"G"**@{--},
(28,0)="G",
"G";"G"+(36,6)*{c_1},
"G";"G"+(72,0)**@{-};"G"+(72,12)**@{-};"G"+(0,12)**@{-};"G"**@{-},
(-4,-2);(104,-2)**@{-},
(12,14)="G",
"G"+(12,6)*{\gamma_2},
"G";"G"+(24,0)**@{--};"G"+(24,12)**@{--};"G"+(0,12)**@{--};"G"**@{--},
(40,14)="G",
"G";"G"+(24,6)*{c_2},
"G";"G"+(48,0)**@{-};"G"+(48,12)**@{-};"G"+(0,12)**@{-};"G"**@{-},
(30,40)*{\vdots},(72,40)*{\vdots},
(30,52)="G",
"G"+(12,6)*{\gamma_n},
"G";"G"+(24,0)**@{--};"G"+(24,12)**@{--};"G"+(0,12)**@{--};"G"**@{--},
(58,52)="G",
"G";"G"+(15,6)*{c_n},
"G";"G"+(30,0)**@{-};"G"+(30,12)**@{-};"G"+(0,12)**@{-};"G"**@{-},
\endxy}
\caption{\usebox{\mycaptionbox}}
  \label{fig:sqpokqlmq}
\end{figure}

\begin{theorem}
  \label{thr:1}
Let $F:\M\to\RR$ be a function, and let $H:\M\to\RR$ be the graded
M\"obius transform of~$F$. Then:
\begin{equation}
  \label{eq:19}
  \forall u\in\M\quad 
F(u)=\sum_{x\in\M(u)}H(x)\,.
\end{equation}
\end{theorem}

\begin{remark}
  \begin{enumerate}
  \item If $\tau(u)=1$, then $u=c$ is a non empty clique. Hence $H(u)$
    coincides with the value $h(u)$, where $h:\C\to\RR$ is the
    M\"obius transform of the
    restriction~$F|_\C$\,. Formula~(\ref{eq:19}) writes as:
    $\sum_{c'\in\C\tq c'\geq c}h(c')=F(c)$, which is the standard
    M\"obius inversion formula~(\ref{eq:17}) for~$F|_\C$\,.
  \item Both the definition of the graded M\"obius transform and the
    inversion formula~(\ref{eq:19}) are valid for functions taking
    values in any commutative group instead of~$\RR$.
  \end{enumerate}
\end{remark}

\begin{proof}[Proof of Theorem~\ref{thr:1}.]
  We first give an alternative formulation for the graded M\"obius
  transform of~$F$, still denoting by $c$ the last clique in the
  Cartier-Foata decomposition of~$u$:
\begin{equation}
  \label{eq:29}
  H(u)=\sum_{\delta\in\C\tq\delta\indep c}(-1)^{|\delta|}F(u\cdot\delta)\,,
\end{equation}
resulting from the change of variable $c'=c\cdot\delta$
in~(\ref{eq:9}).

We now come to the proof of the identity~(\ref{eq:19}). If $u=0$, then
the identity follows from the standard M\"obius inversion
formula~(\ref{eq:17}). 

Hence, let $u\in\M$ be a non empty trace, and
let $c_1\to\ldots\to c_n$ be the Cartier-Foata decomposition of~$u$.
According to the results recalled in~\S~\ref{sec:order-cart-foata}
in~(\ref{eq:20})--(\ref{eq:21}), the Cartier-Foata decomposition of a
generic $x\in\M(u)$ is of the form $d_1\to\ldots\to d_n$ with
$d_i=c_i\cdot\gamma_i$\,, where $(\gamma_1,\ldots,\gamma_n)$ is a
sequence of cliques uniquely determined by~$x$, and such that
  \begin{inparaenum}[(1)] 
    \item $\gamma_i\indep c_i,\ldots,c_n$ for all
      $i\in\{1,\ldots,n\}$, and
    \item $c_1\cdot\gamma_1\to\ldots\to c_n\cdot\gamma_n$ holds.
  \end{inparaenum}
  Consequently, using~(\ref{eq:29}) above, the computation goes as
  follows:
\begin{align}
\notag
  \sum_{x\in
    \M(u)}H(x)&=\sum_{x\in\M(u)}\ \sum_{\delta\in\C\tq\delta\indep
    c_n\cdot\gamma_n}(-1)^{|\delta|}F(x\cdot\delta)\\
\label{eq:31}
&=\sum_{\substack{
\gamma_1,\ldots,\gamma_{n-1}\in\C\;:\\
\gamma_i\indep c_i,\ldots,c_n\text{  for }1\leq i\leq n-1\\
c_1\cdot\gamma_1\to\ldots\to c_{n-1}\cdot\gamma_{n-1}
}}\quad R(\gamma_1,\ldots,\gamma_{n-1})
\\
\intertext{with}
\notag
R(\gamma_1,\ldots,\gamma_{n-1})&=\sum_{\substack{\gamma_n\in\C\tq \gamma_n\indep c_n\,,\\
c_{n-1}\cdot \gamma_{n-1}\to c_n\cdot\gamma_n
}}
\sum_{\delta\in\C\tq\delta\indep
  c_n\cdot\gamma_n}(-1)^{|\delta|}F(c_1\cdot\gamma_1\cdot\ldots\cdot
c_n\cdot\gamma_n\cdot\delta)
\end{align}

The range of the cliques $\gamma_n\in\C$ in the scope of the above sum
is identical to $\gamma_n\indep c_n$ and $c_{n-1}\cdot\gamma_{n-1}\to
\gamma_n$\,, since $c_{n-1}\to c_n$ holds by hypothesis. Using the
change of variable $\gamma=\gamma_n\cdot\delta$ yields:
\begin{align*}
  R(\gamma_1,\ldots,\gamma_{n-1})&=
\sum_{\substack{\gamma_n\in\C\tq \gamma_n\indep c_n\,,\\
c_{n-1}\cdot\gamma_{n-1}\to\gamma_n}}\ 
\sum_{\substack{
\gamma\in\C\tq\gamma\geq\gamma_n\\
\gamma\indep c_n
}}(-1)^{|\gamma|-|\gamma_n|}
 F(c_1\cdot\gamma_1\cdot\ldots\cdot
c_{n-1}\cdot\gamma_{n-1}\cdot c_n\cdot\gamma)\\
&=\sum_{\gamma\in\C\tq\gamma\indep c_n}(-1)^{|\gamma|}
F(c_1\cdot\gamma_1\cdot\ldots\cdot c_{n-1}\cdot\gamma_{n-1}\cdot
c_n\cdot\gamma)K(\gamma)\\
\intertext{with}
K(\gamma)&=\sum_{\substack{\gamma_n\in\C\tq\gamma_n\indep c_n\\
c_{n-1}\cdot\gamma_{n-1}\to\gamma_n\\
\gamma_n\leq\gamma
}}
(-1)^{|\gamma_n|}=\un{\gamma\indep c_{n-1}\cdot\gamma_{n-1}}\quad\text{by the binomial formula.}
\end{align*}

Returning to~(\ref{eq:31}), we obtain thus:
\begin{gather*}
  \sum_{x\in\M(u)}H(x)=\sum_{\substack{\gamma_1,\ldots,\gamma_{n-1}\in\C\;:\\
\gamma_i\indep c_i,\ldots,c_n\text{ for }1\leq i\leq n\\
c_1\cdot\gamma_1\to\ldots\to c_{n-1}\cdot\gamma_{n-1}}}\quad
\sum_{\substack{\delta\in\C\;:\\
\delta\indep c_{n-1}\cdot\gamma_{n-1}\,,\,c_n
}}
(-1)^{|\delta|}F(c_1\cdot\gamma_1\cdot\ldots\cdot
c_{n-1}\cdot\gamma_{n-1}\cdot c_n\cdot\delta)
\end{gather*}

Applying recursively the same transformation eventually yields:
\begin{align*}
  \sum_{x\in\M(u)}H(x)&=
\sum_{\substack{\gamma_1\in\C\;:\\
\gamma_1\indep c_1,\ldots,c_n
}}\quad
\sum_{\substack{
\delta\in\C\;:\\
\delta\indep c_1\cdot\gamma_1,c_2,\ldots,c_n
}}
(-1)^{|\delta|}F(c_1\cdot\gamma_1\cdot c_2\cdot\ldots\cdot
c_n\cdot\delta)\\
&=\sum_{\substack{\delta\in\C\;:\\
\delta\indep c_1,\ldots,c_n}}(-1)^{|\delta|}F(c_1\cdot\delta\cdot
c_2\cdot\ldots\cdot c_n)
\biggl(\sum_{\substack{
\gamma\in\C\;:\\
\gamma\indep c_1,\ldots,c_n\\
\gamma\leq\delta
}}
(-1)^{|\gamma|}
\biggr)\\
&=F(c_1\cdot\ldots\cdot c_n)=F(u)\,.
\end{align*}
The proof is complete.
\end{proof}

\section{M\"obius Harmonic Functions and their Integral Representation}
\label{sec:mobi-harm-funct}

We introduce notations that will be used throughout this section and
the next one. We assume that $\M=\M(\Sigma,I)$ is an irreducible trace
monoid, equipped with a Bernoulli measure $\pr$ on $(\B\M,\FFF)$. We
define the functions $f,h:\M\to\RR$ by letting $f(u)=\pr(\up u)$ for
$u\in\M$, and $h$ is the graded M\"obius transform of~$f$ (see
Definition~\ref{def:1}).

By definition of a Bernoulli measure, the function $f$ is
multiplicative over~$\M$\,: $f(u\cdot v)=f(u)f(v)$. Therefore, if
$u\in\M$ is such that $u=v\cdot c$, with $c$ the last clique in the
Cartier-Foata decomposition of~$u$, it follows from
Definition~\ref{def:1} that $h(u)=f(v)h(c)$. Hence, according
to~(\ref{eq:13}), if $c_1\to\ldots\to c_n$ are $n\geq1$ non empty
cliques and $u=c_1\cdot\ldots\cdot c_n$\,, one has:
\begin{equation}
  \label{eq:22}
  \pr(C_1=c_1,\ldots,C_n=c_n)=h(u).
\end{equation}

\begin{definition}
  \label{def:2}
A \emph{M\"obius harmonic} function, relative to a pair $(\M,\pr)$ as above,
is a function $\lambda:\M\to\RR$ such that:
\begin{equation}
  \label{eq:23}
  \forall u\in\M\quad
\sum_{c\in\C}(-1)^{|c|}f(c)\lambda(u\cdot c)=0\,.
\end{equation}
\end{definition}

Obviously, M\"obius harmonic functions form a real vector space.  The
first example of M\"obius harmonic functions are constant
functions. Indeed, if $\lambda=1$ identically on~$\M$, then
(\ref{eq:23})~reduces to:
\begin{equation*}
  \sum_{c\in\C}(-1)^{|c|}f(c)=0\,,
\end{equation*}
which holds since we recognize the M\"obius transform $h$ evaluated
at~$0$ in the above expression, and $h(0)=0$ by~(\ref{eq:12}).

Another way to obtain bounded M\"obius harmonic functions is given by
the next result. We denote by $L^\infty(\B\M)$ the space of functions
$\varphi:\B\M\to\RR$ bounded \mbox{$\pr$-modulo~$0$}.

\begin{proposition}
\label{prop:2}
  For every $\varphi\in L^\infty(\B\M)$, the function
  $\lambda:\M\to\RR$ defined by:
  \begin{equation}
    \label{eq:24}
    \forall u\in\M\quad \lambda(u)=\frac1{f(u)}\int_{\up u}\varphi(\xi)\,d\pr(\xi)
  \end{equation}
is M\"obius harmonic and bounded on~$\M$.
\end{proposition}

\begin{proof}
  It is obvious that $\lambda$ thus defined is bounded on~$\M$
  by~$\|\varphi\|_\infty$\,. Let $u\in\M$ be a trace, and consider the
  following non disjoint union:
\begin{equation}
  \label{eq:32}
  \up u=\bigcup_{a\in\Sigma}\up(u\cdot a)=\bigcup_{i=1}^k\up(u\cdot a_i)\,,
\end{equation}
where $\Sigma=\{a_1,\ldots,a_k\}$ is an enumeration of~$\Sigma$.  We
decompose the integral in~(\ref{eq:24}) with respect to the
union~(\ref{eq:32}), and using Poincar\'e inclusion-exclusion
principle:
\begin{align*}
  f(u)\lambda(u)&=\sum_{r=1}^k(-1)^{r+1}\sum_{1\leq i_1<\cdots<i_r\leq
  k}\int_{\up(u\cdot a_{i_1})\cap\cdots\cap \up(u\cdot a_{i_r})}\varphi\,d\pr\,.
\end{align*}

An intersection $\up(u\cdot a_{i_1})\cap\cdots\cap\up(u\cdot a_{i_r})$
is empty unless $\{a_{i_1},\ldots,a_{i_r}\}$ is a clique, in which
case the intersection coincides with $\up(u\cdot a_{i_1}\cdot\ldots\cdot
a_{i_r})$. Henceforth the above sum evaluates as:
\begin{align*}
  f(u)\lambda(u)&=\sum_{c\in\Cstar}(-1)^{|c|+1}\int_{\up(u\cdot
    c)}\varphi\,d\pr\,,&
\text{with }\Cstar&=\C\setminus\{0\}\,.
\end{align*}

Introducing $f(u\cdot c)$ in the above sum in order to recognize
$\lambda(u\cdot c)$ yields:
\begin{align*}
  f(u)\lambda(u)+\sum_{c\in\Cstar}(-1)^{|c|}f(u\cdot c)\lambda(u\cdot c)=0\,.
\end{align*}
Since $f$ is multiplicative and positive on~$\M$, we simplify by
$f(u)$, and recognize in $\lambda(u)$ the missing term of the sum for
$c=0$, yielding:
\begin{align*}
  \sum_{c\in\C}(-1)^{|c|}f(c)\lambda(u\cdot c)=0\,,
\end{align*}
which was to be proved.
\end{proof}

Our goal is now to prove a converse for Proposition~\ref{prop:2};
hence, starting from a bounded M\"obius harmonic function
$\lambda:\M\to\RR$, to find $\varphi\in L^\infty(\B\M)$ such
that~(\ref{eq:24}) holds. The remaining of the section is devoted to
the proof of this result, stated in Theorem~\ref{thr:2} below.

From now on, we fix a bounded M\"obius harmonic function
$\lambda:\M\to\RR$. Our method of proof to find the adequate
$\varphi\in L^\infty(\B\M)$ follows loosely the same line of proof
than for harmonic functions on trees for
instance~\cite{mouton00}. However, the issue here is to find an
adequate martingale, the expression of which is not obvious \textit{a
  priori}.

Reasoning by analysis, assume first that $\varphi$ exists. Recall that
we have defined in~\S~\ref{sec:comb-trace-mono} the sequence of random
variables~$(C_n)_{n\geq1}$\,, given by the non empty cliques of the
Cartier-Foata decomposition of a generic element $\xi\in\B\M$. For
each integer $n\geq1$, consider the sub-\slgb\ of $\FFF$ generated by $(C_1,\ldots,C_n)$:
\begin{equation*}
  \FFF_n=\sigma\langle
  C_1,\ldots,C_n\rangle\,.  
\end{equation*}

Obviously, $(\FFF_n)_{n\geq1}$ forms a filtration of~$\FFF$, and
therefore the sequence of conditional expectations
$\bigl(\esp(\varphi|\FFF_n),\FFF_n\bigr)_{n\geq1}$ is a
martingale. Putting $Y_n=\esp(\varphi|\FFF_n)$\,, the problem that we
face is to express $Y_n$ using $\lambda$ only. Since $Y_n$ is
$\FFF_n$-measurable, $Y_n$~can be seen as a function of the $n$ first
cliques $C_1,\ldots,C_n$\,.  Some computations, the details of which
will be given in the proof of Theorem~\ref{thr:2} below, lead to the
following potential form for~$Y_n$:
\begin{equation}
    \label{eq:25}
Y_n=\frac1{h(C_n)}\sum_{c\in\C\tq c\geq
  C_n}(-1)^{|c|-|C_n|}f(c)\lambda(V\cdot c)\,,\text{ with
}V=C_1\cdot\ldots\cdot C_{n-1}\,.
\end{equation}

Based on the above analysis, we are naturally brought to prove the
following result.

\begin{lemma}
\label{lem:1}
Let $\lambda:\M\to\RR$ be a bounded M\"obius harmonic function. For
each integer $n\geq1$, let $Y_n$ be the $\FFF_n$-measurable random
variable defined by~\eqref{eq:25}.  Then $(Y_n,\FFF_n)_{n\geq1}$ is a
bounded martingale.
\end{lemma}

\begin{proof}
  It is obvious that $Y_n$ is $\FFF_n$-measurable, bounded and thus
  integrable. Hence we only have to show that
  $E(Y_n|\FFF_{n-1})=Y_{n-1}$ holds for all $n\geq2$\,. Putting
  $Z_n=E(Y_n|\FFF_{n-1})$, we evaluate $Z_n$ on the atom
  $\{C_1=c_1,\ldots,C_{n-1}=c_{n-1}\}$ with $c_1\to\ldots\to
  c_{n-1}$\,, by computing as follows, and putting
  $v=c_1\cdot\ldots\cdot c_{n-1}$\,:
  \begin{align*}
    Z_n&=\sum_{c\in\Cstar\tq c_{n-1}\to
      c}\pr(C_n=c|C_1=c_1,\ldots,C_{n-1}=c_{n-1})Y_n(c_1,\ldots,c_{n-1},c)\\
&=\sum_{c\in\Cstar\tq c_{n-1}\to c}
\frac{h(c)}{g(c_{n-1})}\frac1{h(c)}\sum_{c'\in\C\tq c'\geq
  c}(-1)^{|c'|-|c|}f(c')\lambda(v\cdot c')
\intertext{where $g$ is the normalization factor defined
  in~(\ref{eq:26}),}
&=\frac1{g(c_{n-1})}\sum_{c'\in\C}(-1)^{|c'|}f(c')\lambda(v\cdot
c')\underbrace{\sum_{c\in\Cstar\tq c\leq c'\wedge c_{n-1}\to c}(-1)^{|c|}}_{K(c',c_{n-1})}
  \end{align*}

  Using the binomial formula, we have for any two cliques $d$ and~$d'$:
  \begin{equation*}
    \sum_{\delta\in\C\tq\delta\leq d'\wedge d\to\delta}(-1)^{|\delta|}=
    \begin{cases}
1,&\begin{array}[t]{l}%
\text{if there is no $\delta\leq d'$ but $0$ such that
     $d\to\delta$}\\
\quad\iff d\indep d' 
\end{array}
\\
0,&\begin{array}[t]{l}\text{otherwise}
\end{array}
\end{cases}
  \end{equation*}
Hence:
 \begin{equation*}
   K(c',c_{n-1})=-1+\sum_{c\in\C\tq c\leq c'\wedge c_{n-1}\to
     c}(-1)^{|c|}=-\un{\neg(c'\indep c_{n-1})}\,.
 \end{equation*}
All put together, this yields:
\begin{equation}
  \label{eq:28}
  Z_n=\frac{-1}{g(c_{n-1})}\sum_{c\in\C\tq\neg(c\indep c_{n-1})}(-1)^{|c|}f(c)\lambda(v\cdot
  c)\,.
\end{equation}

 According to the M\"obius harmonicity of $\lambda$ at~$v$, one has:
 \begin{equation*}
   \sum_{c\in\C}(-1)^{c}f(c)\lambda(v\cdot c)=0\,.
 \end{equation*}

Decomposing $\C$ into those $c\in\C$ such that
$c\indep c_{n-1}$ and those $c\in\C$ such that $\neg(c\indep c_{n-1})$,
and re-injecting in~(\ref{eq:28}) yields:
\begin{align}
\notag
Z_n&=\frac1{g(c_{n-1})}\sum_{c\in\C \tq c\indep
  c_{n-1}}(-1)^{|c|}f(c)\lambda(v\cdot c)\,,\\
\intertext{and with the change of variable $\delta=c\cdot c_{n-1}$:}
\label{eq:27}
Z_n&=\frac1{g(c_{n-1})}
\sum_{\delta\in\C\tq\delta\geq c_{n-1}}
(-1)^{|\delta|-|c_{n-1}|}\frac{f(\delta)}{f(c_{n-1})}
\lambda(w\cdot\delta)\,,
\end{align}
where $w=c_1\cdot\ldots\cdot c_{n-2}$\,. But
$g(c_{n-1})f(c_{n-1})=h(c_{n-1})$, as recalled
in~(\ref{eq:30}). Henceforth, Equation~(\ref{eq:27}) writes as
$Z_n=Y_{n-1}$, which was to be shown.
\end{proof}

Since the sequence $(Y_n,\FFF_n)_{n\geq1}$ is a bounded martingale, it
converges \pas\ and in the space $L^1(\B\M)$ to a limit $\varphi\in
L^\infty(\B\M)$. It is natural to expect that this limit is the
adequate candidate for the integral representation of~$\lambda$. This
is true indeed, and the proof of this fact is based on the inversion
formula for graded M\"obius transforms proved
in~\S~\ref{sec:prel-comb-results} above. Therefore Theorem~\ref{thr:2}
below provides a probabilistic interpretation of the inversion
formula.

\begin{theorem}
  \label{thr:2}
For every bounded M\"obius harmonic function $\lambda:\M\to\RR$, there
exists a unique $\varphi\in L^\infty(\B\M)$ such that:
\begin{equation}
  \label{eq:33}
  \forall u\in\M\quad
\lambda(u)=\frac1{\pr(\up u)}\int_{\up u}\varphi\,d\pr\,.
\end{equation}
The above formula establishes a bijective and isometric linear
correspondence between $L^\infty(\B\M)$ and the space of bounded
M\"obius harmonic functions on~$\M$. Both this correspondence and its
inverse are positive operators.
\end{theorem}

\begin{proof}
  Let $\lambda:\M\to\RR$ be a bounded M\"obius harmonic function.  Let
  $\varphi\in L^\infty(\B\M)$ be the limit, \pas\ and in~$L^1(\B\M)$,
  of the martingale $(Y_n,\FFF_n)_{n\geq1}$ defined as in
  Lemma~\ref{lem:1}. We prove
  that~(\ref{eq:33}) holds for this function~$\varphi$. 

  Let $u\in\M$ be a trace.  To compute the integral of $\varphi$
  over~$\up u$, we rely on the description~(\ref{eq:14}) of $\up u$
  stated in~\S~\ref{sec:order-cart-foata}. Let $n=\tau(n)$ be the
  height of~$u$. Then, by~(\ref{eq:14}), and denoting as in the proof
  of Theorem~\ref{thr:1}:
  \begin{align*}
    \M(u)=\{x\in\M\tq\tau(x)=n\wedge u\leq x\}\,,
  \end{align*}
one has:
  \begin{align*}
    \up u=\{\xi\in\B\M\tq C_1\cdot\ldots\cdot C_n\geq u\}
=\bigcup_{x\in\M(u)}\{\xi\in\B\M\tq C_1\cdot\ldots\cdot C_n=x\}\,,
  \end{align*}
the last union being disjoint. Accordingly, and using
formula~(\ref{eq:22}) found above, one has:
\begin{align*}
  \int_{\up u}\varphi\,d\pr&=\sum_{x\in\M(u)}\int_{C_1\cdot\ldots\cdot
  C_n=x}\varphi\,d\pr=\sum_{x\in\M(u)}h(x)\esp(\varphi|C_1\cdot\ldots\cdot C_n=x)\,.
\end{align*}
By definition of the conditional expectation
$Y_n=\esp(\varphi|\FFF_n)$, using the expression~(\ref{eq:25}) and
since $\FFF_n=\sigma\langle C_1,\ldots,C_n\rangle$, we deduce:
 \begin{align}
 \label{eq:35}
   \int_{\up u}\varphi\,d\pr&=
 \sum_{x\in\M(u)}f(y)\sum_{c\in\C\tq c\geq
   \gamma_n}(-1)^{|c|-|\gamma_n|}f(c)\lambda(y\cdot c)\,,
 \end{align}
where $x=y\cdot \gamma_n$ is the decomposition of a generic element
$x\in\M(u)$ such that $\gamma_n$ is the last clique in the
Cartier-Foata decomposition of~$x$.

Let $F:\M\to\RR$ be the function defined by $F(u)=f(u)\lambda(u)$, and
let $H$ be the graded M\"obius transform of~$F$ (see
Definition~\ref{def:1}). Since $f$ is multiplicative,
Equation~(\ref{eq:35}) writes as:
 \begin{align*}
   \int_{\up u}\varphi\,d\pr&=\sum_{x\in\M(u)}H(x)\\
 &=F(u)&\text{by Theorem~\ref{thr:1}}\\
 &=f(u)\lambda(u)\,.
 \end{align*}
This shows formula~(\ref{eq:33}).

\medskip We have shown the existence of~$\varphi$, and we now focus on
its uniqueness. It is enough to show that $\varphi=\lim_{n\to\infty}
Y_n$ necessarily holds \pas\ if $\varphi\in L^\infty(\B\M)$
satisfies~(\ref{eq:33}). Hence, consider $\varphi\in
L^\infty(\B\M)$. The limit
$\varphi=\lim_{n\to\infty}\esp(\varphi|\FFF_n)$ holds \pas, since
$\bigvee_{n\geq1}\FFF_n=\FFF$\,, as attested by~(\ref{eq:14}). We are
thus bound to prove $\esp(\varphi|\FFF_n)=Y_n$\,. For this, we compute
as follows, considering a sequence $c_1\to\ldots\to c_n$ of non empty
cliques, and putting $u=c_1\cdot\ldots\cdot c_n$ and
$v=c_1\cdot\ldots\cdot c_{n-1}$ (this is the computation we promised just
above Lemma~\ref{lem:1}):
\begin{align}
\notag
  \esp(\varphi|C_1=c_1,\ldots,C_n=c_n)&=\frac1{h(u)}\int_{C_1=c_1,\ldots,C_n=c_n}\varphi\,d\pr\\
\label{eq:34}
&=\frac1{h(u)}\biggl(\int_{\up u}\varphi\,d\pr-
\underbrace{\int_{\bigcup_{c\in\C\tq
    c>c_n}\up(v\cdot c)}\varphi\,d\pr}_{L}\ 
\biggr)\,,
\end{align}
the later equality according to~(\ref{eq:15}). 

Let $\{a_1,\ldots,a_k\}$ be an enumeration of those letters
$a\in\Sigma$ such that $a\indep c_n$. Then it is obvious that:
\begin{equation*}
  \bigcup_{c\in\C\tq c>c_n}\up(v\cdot c)=\bigcup_{i=1}^k\up(v\cdot c_n\cdot a_i)\,.
\end{equation*}

Applying Poincar\'e inclusion-exclusion principle as in the proof of
Proposition~\ref{prop:2}, we deduce the following expression
for~$L$ defined in~(\ref{eq:34}) above:
\begin{align*}
  L&=\sum_{r=1}^k(-1)^{r+1}\sum_{1\leq i_1<\ldots<i_r\leq
    k}\int_{\up(v\cdot c_n\cdot a_{i_1})\cap\cdots\cap\up(v\cdot c_n\cdot
    a_{i_r})}\varphi\,d\pr\\
&=\sum_{c\in\C\tq c>c_n}(-1)^{|c|+1}\int_{\up(v\cdot c)}\varphi\,d\pr\,.
\end{align*}

Returning to~(\ref{eq:34}), we obtain:
\begin{gather*}
  \esp(\varphi|C_1=c_1,\ldots,C_n=c_n)=\frac1{h(u)}\sum_{c\in\C\tq
  c\geq c_n}(-1)^{|c|-|c_n|}f(v\cdot c)\lambda(v\cdot c)\,.
\end{gather*}

But $h(u)=f(v)h(c_n)$ and $f(v\cdot c)=f(v)f(c)$, hence we obtain the
expected expression~(\ref{eq:25}) defining $Y_n$ for
$\esp(\varphi|\FFF_n)$. This proves the uniqueness of~$\varphi$. 

\medskip Let $\MH^\infty(\M)$ denote the linear space of bounded
M\"obius harmonic functions of~$\M$, and let
$\Psi:L^\infty(\B\M)\to\MH^\infty(\M)$ be the transformation defined
by~(\ref{eq:33}). It is obvious that $\Psi$ is linear, and we have
shown that $\Psi$ is bijective; it remains only to show that $\Psi$
and $\Psi^{-1}$ are positive and isometric.

It is obvious on the expression~(\ref{eq:33}) that $\Psi$ is a
positive operator ($\varphi\geq0\implies\lambda\geq0$). And since
$f=\pr(\up\cdot)>0$ on $\M$ by assumption, the fact that $\Psi^{-1}$
is also positive follows from Lemma~\ref{lem:2} below.  Since $\Psi$
and $\Psi^{-1}$ are positive operators, and since $\Psi(1)=1$, it is
an easy consequence that they are isometric.
\end{proof}

\begin{lemma}
  \label{lem:2}
  If\/ $\varphi\in L^\infty(\B\M)$ is such that $\int_{\up
    u}\varphi\,d\pr\geq0$ holds for all \mbox{$u\in\M$}, then
  $\varphi\geq0$ holds\, \pas\ on~$\B\M$.
\end{lemma}

\begin{proof}
  The collection of elementary cylinders, to which is added the empty
  set, is stable by finite intersections. Hence the lemma is an
  application of the Monotone Class theorem, since $\FFF$ is the
  \slgb\ generated by all elementary cylinders~$\up u$, for $u\in\M$.
\end{proof}

\begin{corollary}
  \label{cor:1}
Let $\lambda:\M\to\RR$ be a bounded and non negative M\"obius harmonic
function. Then for any trace $u\in\M$, if $c$ is the last clique in
the Cartier-Foata decomposition of~$u$, one has:
\begin{equation}
\label{eq:44}
  \sum_{\delta\in\C\tq\delta\indep c}(-1)^{|\delta|}f(\delta)\lambda(u\cdot\delta)\geq0.
\end{equation}
\end{corollary}

\begin{remark}
  The result of the corollary is not obvious, because of the presence
  of negative terms in the sum. We shall see
  in~\S~\ref{sec:non-negative-mobius} below an example
  where~(\ref{eq:44}) does not hold for a non negative
  \emph{unbounded} M\"obius harmonic function.
\end{remark}

\begin{proof}[Proof of Corollary~\ref{cor:1}.]
  Let $u\in\M$ and $\lambda:\M\to\RR$ be as in the statement, and let
  $\varphi\in L^\infty(\B\M)$ be associated to $\lambda$ as in
  Theorem~\ref{thr:2}. Let $c_1\to\ldots\to c_n$ be the Cartier-Foata
  decomposition of~$u$, and put $v=c_1\cdot\ldots\cdot c_{n-1}$\,.
 
  Then Theorem~\ref{thr:2} states that $\varphi\geq0$ holds \pas\
  on~$\B\M$. Therefore $Y_n=\esp(\varphi|\FFF_n)$ is \pas\ non
  negative on~$\B\M$. Evaluating $Y_n$ on the atom
  $\{C_1=c_1,\ldots,C_n=c_n\}$ of~$\FFF_n$ according to the
  expression~(\ref{eq:25}) for~$Y_n$, which was derived in the course
  of the proof of Theorem~\ref{thr:2}, yields:
  \begin{gather*}
\frac1{h(c_n)}    \sum_{c\in\C\tq c\geq
  c_n}(-1)^{|c|-|c_n|}f(c)\lambda(v\cdot c_n)\geq0\,.
  \end{gather*}

  Since $h>0$ on $\Cstar$ by~(\ref{eq:12}), and since $f$ is
  multiplicative, the result follows from the change of variable
  $c=c_n\cdot\delta$ in the above sum.
\end{proof}

\section{Additional Remarks}
\label{sec:non-negative-mobius}

In this section, we examine some examples of unbounded M\"obius
harmonic functions that arise naturally. We consider as above a pair
$(\M,\pr)$, where $\M$ is an irreducible trace monoid, and $\pr$ is a
Bernoulli measure on~$(\B\M,\FFF)$.  As usual we put $f(u)=\pr(\up
u)$ for $u\in\M$, and $h:\M\to\RR$ is defined as the graded M\"obius
transform of~$f$. 

The first observation is that, if $\nu$ is any finite measure on
$(\B\M,\FFF)$, then the function $\lambda:\M\to\RR$ defined by:
\begin{equation}
  \label{eq:37}
  \forall u\in\M\quad \lambda(u)=\frac1{\pr(\up u)}\nu(\up u)\,,
\end{equation}
is M\"obius harmonic. The proof is similar to the proof of
Proposition~\ref{prop:2}. It is also a reformulation of
Proposition~2.1 of~\cite{abbesmair14}. Note that
applying~(\ref{eq:37}) to the measure $d\nu=\varphi d\pr$ brings back
the result of Proposition~\ref{prop:2} on the M\"obius harmonicity
of~$\lambda$. Contrary to the result of Proposition~\ref{prop:2}
however, in general the function $\lambda$ defined
in~(\ref{eq:37}) is unbounded.

The \emph{Green kernel} of $(\M,\pr)$ is defined by:
\begin{equation}
  \label{eq:38}
  \forall x,y\in\M\quad
G(x,y)=
\begin{cases}
 f(y)/f(x),&\text{if $x\leq y$,}\\
0,&\text{otherwise.}
\end{cases}
\end{equation}

The fact that our basic object $\M$ is a semi-group rather than a
group makes that $G$ is not positive on~$\M\times\M$\,.  For $y\in\M$,
put $G_y=G(\,\cdot\,,y)$, and define $\Delta G_y:\M\to\RR$ by:
\begin{equation}
  \label{eq:39}
\forall x\in\M\quad  \Delta G_y(x)=\sum_{c\in\C}(-1)^{|c|}f(c)G_y(x\cdot
c)\,.
\end{equation}

Easy calculations show that:
\begin{equation}
  \label{eq:40}
\forall x,y\in\M\quad  \Delta G_y(x)=
\begin{cases}
  1,&\text{if $x=y$,}\\
0,&\text{otherwise}.
\end{cases}
\end{equation}

In other words, extending in the obvious way Definition~\ref{def:1}
to functions M\"obius harmonic on a subset of~$\M$,
$G_y$~is M\"obius harmonic on $\M\setminus\{y\}$\,.  Therefore $y$
appears as the unique singularity of~$G_y$\,. The standard idea from
Martin theory is to ``send the singularity at infinity''. We define
the \emph{Martin kernels}~$K_y$\,, for $y$ ranging over~$\M$, as
follows:
\begin{equation}
  \label{eq:41}
  \forall y\in\M\quad
  K_y=\frac{G(\,\cdot\,,y)}{G(0,y)}=\un{\,\cdot\,\leq y}\frac1{f(\,\cdot\,)}\,.
\end{equation}

Sending $y$ to infinity consists in taking a limit along a sequence
$(y_n)_{n\geq1}$ of traces converging to a point $\xi\in\B\M$ of the
boundary. The topological framework on $\Mbar=\M\cup\B\M$ does not
present any particular difficulty; the easiest way is to simply
identify generalized traces with their Cartier-Foata decomposition,
and to use the standard metric constructions on sequences, either
finite or infinite, taking values in the finite set~$\Cstar$. Hence we
use this topological framework without stating more formal
definitions. The space of M\"obius harmonic functions is then endowed
with the pointwise convergence; any limit of M\"obius harmonic
functions is M\"obius harmonic.

Within this framework, it is visible on~(\ref{eq:41}) that, if
$y_n\to\xi\in\B\M$, then $(K_{y_n})_{n\geq1}$ converges to
$K_\xi:\M\to\RR$\,, defined by:
\begin{equation}
  \label{eq:42}
\forall \xi\in\B\M\quad  \forall x\in\M\quad K_\xi(x)=\frac1{f(x)}\un{x\leq\xi}\,,
\end{equation}
which is M\"obius harmonic, this time on~$\M$. The Martin kernel
$K_\xi$ thus defined corresponds to the harmonic function defined as
in~(\ref{eq:37}) with respect to the Dirac measure $\delta_\xi$
concentrated on~$\xi$. It is obviously unbounded.

In general, if $\nu$ is a finite measure on $\B\M$ such that the
associated M\"obius harmonic function $\lambda$ is bounded on~$\B\M$,
then $\nu$ is regular with respect to~$\pr$ and $\varphi\in
L^\infty(\B\M)$ associated to $\lambda$ by Theorem~\ref{thr:2}
coincides \pas\ with the Radon-Nykodim derivative~$d\nu/d\pr$\,. This
is obviously not the case for the Dirac measures~$\delta_\xi$\,.

\medskip Does every non negative M\"obius harmonic function
$\lambda:\M\to\RR$ originate from a---necessarily finite---measure
$\nu$ on~$\B\M$ as in~(\ref{eq:37})?  The answer is negative, as the
following example reveals.

Assume that $\pr$ is the \emph{uniform Bernoulli measure} on~$\B\M$,
which is defined by $\pr(\up u)=p_0^{|u|}$ for every trace $u\in\M$,
where $p_0$ is the unique root of smallest modulus of the M\"obius
polynomial~$\mu_\M(X)$ (see~\S~\ref{sec:mobi-transf-mobi}). Let $p$ be
another non negative root of~$\mu_\M(X)$, if it exists, and define
$\lambda:\M\to\RR$ by:
\begin{equation}
  \label{eq:43}
  \forall u\in\M\quad \lambda(u)=\Bigl(\frac p{p_0}\Bigr)^{|u|}\,.
\end{equation}

Then $\lambda$ is M\"obius harmonic, and it is clearly unbounded since
$p>p_0$\,. We claim that \emph{there exists no finite measure $\nu$ on
  $(\B\M,\FFF)$ such that~\eqref{eq:37} would hold for $\lambda$
  and~$\nu$.}

By contradiction, assume that $\nu$ exists. Then we would have
$\nu(\up u)=p^{|u|}$ for all $u\in\M$, and in particular $\nu$ would
be a probability measure assigning an equal probability to all
cylinders $\up u$ for $u$ ranging over traces of a fixed length. But
then, it follows from the---rather difficult---result
of~\cite[Th.~5.1~point~2]{abbesmair14} that $\nu(\up u)=p_0^{|u|}$
for all $u\in\M$, a contradiction. The claim is proved.

It remains only to check that there exists irreducible trace monoids
with different real non negative roots of their M\"obius
polynomials. It is easy to find such examples. Consider for instance
the trace monoid generated by $(\Sigma,I)$ depicted on
Figure~\ref{fig:qqka,pq}---it was already worked out
in~\cite[\S6]{abbesmair14}. Then $\mu_\M(X)=1-5X+5X^2$ has the two
roots $p_0=\frac12-\frac{\sqrt5}{10}$ and
$p_1=\frac12+\frac{\sqrt5}{10}$\,. The uniform Bernoulli measure on
$\B\M$ is characterized by $\pr(\up u)=p_0^{|u|}$\,, for all
$u\in\M$\,. This example is specially interesting since the seond
root~$p_1$ lies itself in the interval~$(0,1)$. Henceforth the
function $u\in\M\mapsto p_1^{|u|}$ satisfies $u\leq v\implies
p_1^{|v|}\leq p_1^{|u|}$; whereas for values $p>1$, this is the
reverse ordering, which disqualifies $p^{|u|}$ at once for being
represented as $p^{|u|}=\nu(\up u)$ for any measure~$\nu$.

\begin{figure}
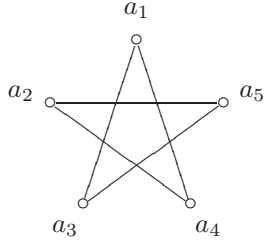

  \centering
  \begin{equation*}
  \xy
{\xypolygon5"A"{~:{(2,2):(3,3):}~={0}~><{@{}}{\circ}
}%
},%
{\xypolygon5"B"{~:{(2,2):(4,4):}~><{@{}}~={0}{a_{\xypolynode}}}},
"A1";"A3"**@{-};"A5"**@{-};"A2"**@{-};"A4"**@{-};"A1"**@{-}
\endxy
  \end{equation*}  
  \caption{\textsl{An independence graph $(\Sigma,I)$}}
  \label{fig:qqka,pq}
\end{figure}

\medskip Typically, this example provides a non negative M\"obius
harmonic function $\lambda(u)=(p_1/p_0)^{|u|}$ such that~(\ref{eq:44})
does not hold. Indeed, evaluating the left hand member
of~(\ref{eq:44}) at trace $u=a_1$ for
$\lambda(u)=(p_1/p_0)^{|u|}$ yields:
  \begin{equation*}
\lambda(a_1)-f(a_3)\lambda(a_1\cdot a_3)-f(a_4)\lambda(a_1\cdot a_4)=    \frac{p_1}{p_0}(1-2p_1)<0\,.
  \end{equation*}

\section{Conclusion}
\label{sec:conclusion}

This work suggests extensions in different directions. First, the
graded M\"obius transform is likely to extend to finitely presented
monoids with an adequate normal form for their elements, typically
finite type Coxeter monoids, including braid monoids.  The extension
to these monoids of the probabilistic framework of Bernoulli measures
and of M\"obius harmonicity is a reasonable target. Dealing with
groups rather than monoids is a non trivial extension, since the
partial order structure collapses.

Second, pursuing the elements of a potential theory for Bernoulli
measures, either in the framework of trace monoids or in a more
general framework, is also natural. In particular, the notion of
super-M\"obius harmonic functions has a natural definition.  A Green
representation of super-M\"obius harmonic functions seems to arise
naturally.

Finally, establishing a bridge with the theory of Poisson-Furstenberg
boundary seems to be an interesting task, despite the first
obstruction mentioned in the Introduction: M\"obius harmonic functions
are not invariant with respect to an obvious Markov operator.

\noindent
\small

\bibliographystyle{plain}
\bibliography{harmonic}

\end{document}

%% file: harmonic.bbl
\begin{thebibliography}{10}

\bibitem{abbes08}
S.~Abbes.
\newblock On countable completions of quotient ordered semigroups.
\newblock {\em Semigroup Forum}, 3(77):482--499, 2008.

\bibitem{abbesmair14}
S.~Abbes and J.~Mairesse.
\newblock {Uniform and Bernoulli measures on the boundary of trace monoids}.
\newblock \texttt{arXiv 1407.5879}, July 2014.
\newblock \url{http://arxiv.org/abs/1407.5879}. Submitted for publication.

\bibitem{cartier72}
P.~Cartier.
\newblock Fonctions harmoniques sur un arbre.
\newblock In {\em Symposia Mathematica}, volume~IX, pages 203--270. Academic
  Press, 1972.

\bibitem{cartier69}
P.~Cartier and D.~Foata.
\newblock {\em Probl\`emes combinatoires de commutation et r\'earrangements},
  volume~85 of {\em Lecture Notes in Mathematics}.
\newblock Springer, 1969.

\bibitem{csikvari13}
P.~Csikv\'ari.
\newblock Note on the smallest root of the independence polynomial.
\newblock {\em Combinatorics, Probability and Computing}, 22(1):1--8, 2013.

\bibitem{diekert90}
V.~Diekert.
\newblock {\em Combinatorics on Traces}, volume 454 of {\em Lecture Notes in
  Computer Science}.
\newblock Springer, 1990.

\bibitem{fisher89}
D.C. Fisher.
\newblock The number of words of length $n$ in a graph monoid.
\newblock {\em The American Mathematical Montly}, 96(7):610--614, 1989.

\bibitem{goldwurm00}
M.~Goldwurm and M.~Santini.
\newblock Clique polynomials have a unique root of smallest modulus.
\newblock {\em Information Processing Letters}, 75(3):127--132, 2000.

\bibitem{kaimanovich96:}
V.A Kaimanovich.
\newblock {Boundaries of invariant Markov operators: the identification
  problem}.
\newblock In M.~Pollicott and K.~Schmidt, editors, {\em Ergodic Theory of
  $\mathbb{Z}^d$-actions, Proceedinds of Warwick Symposium 1993--94}, volume
  228 of {\em London Math. Soc. Lecture Note Series}, pages 127--176. Cambridge
  University Press, 1996.

\bibitem{kaimanovich00}
V.A. Kaimanovich.
\newblock {The Poisson formula for groups with hyperbolic properties}.
\newblock {\em Annals of Mathematics. Second Series}, 152(3):659--692, 2000.

\bibitem{krob03}
D.~Krob, J.~Mairesse, and I.~Michos.
\newblock Computing the average parallelism in trace monoids.
\newblock {\em Discrete Mathematics}, 273:131--162, 2003.

\bibitem{levit05}
V.E. Levit and E.~Mandrescu.
\newblock The independence polynomial of a graph -- a survey.
\newblock In {\em Proceedings of the First International Conference on
  Algebraic Informatics}, pages 233--254. Aristitle University of Thessaloniki,
  2005.

\bibitem{malyutin05}
A.V. Malyutin.
\newblock {The Poisson-Furstenberg boundary of a locally free group}.
\newblock In {\em Representation theory, dynamical systems, combinatorial and
  algorithmic methods. Part~IX}, volume 301 of {\em Zap. Nauchn. Sem. POMI},
  pages 195--211. POMI, St.~Petersburg, 2003.
\newblock English transl.: Journal of Mathematical Sciences 129(2): 3787--3795,
  2005.

\bibitem{mouton00}
F.~Mouton.
\newblock Comportement asymptotique des fonctions harmoniques sur les arbres.
\newblock In {\em S\'eminaire de Probabilit\'es}, volume XXXIV, pages 353--373.
  Universit\'e de Strasbourg, 2000.

\bibitem{rota64}
G.-C. Rota.
\newblock {On the foundations of combinatorial theory~I. Theory of M\"obius
  functions}.
\newblock {\em Z. Wahrscheinlichkeitstheorie}, 2:340--368, 1964.

\bibitem{VNBi}
A.~Vershik, S.~Nechaev, and R.~Bikbov.
\newblock Statistical properties of locally free groups with applications to
  braid groups and growth of random heaps.
\newblock {\em Communications in Mathematical Physics}, 212(2):469--501, 2000.

\bibitem{vershik00}
A.M. Vershik.
\newblock Dynamic theory of growth in groups: entropy, boundaries, examples.
\newblock {\em Uspekhi Mat. Nauk}, 55(4):59--128, 2000.
\newblock English transl.: Russian Math. Surveys 55(4):667--733, 2000.

\bibitem{viennot86}
X.~Viennot.
\newblock Heaps of pieces,~{I} : basic definitions and combinatorial lemmas.
\newblock In {\em Combinatoire \'enum\'erative}, volume 1234 of {\em Lecture
  Notes in Mathematics}, pages 321--350. Springer, 1986.

\end{thebibliography}
